\documentclass[11pt]{article}

\usepackage{amsmath}
\usepackage{amsfonts}
\usepackage{color}
\usepackage{amssymb}
\usepackage{graphicx}
\usepackage{hyperref}
\usepackage{enumerate}
\usepackage[all]{xy}

 \allowdisplaybreaks

\begin{document}

\newtheorem{problem}{Problem}

\newtheorem{theorem}{Theorem}[section]
\newtheorem{corollary}[theorem]{Corollary}
\newtheorem{definition}[theorem]{Definition}
\newtheorem{conjecture}[theorem]{Conjecture}
\newtheorem{question}[theorem]{Question}
\newtheorem{lemma}[theorem]{Lemma}
\newtheorem{proposition}[theorem]{Proposition}
\newtheorem{quest}[theorem]{Question}
\newtheorem{example}[theorem]{Example}

\newenvironment{proof}{\noindent {\bf
Proof.}}{\rule{2mm}{2mm}\par\medskip}

\newenvironment{proofof3}{\noindent {\bf
Proof of  Theorem 1.2.}}{\rule{2mm}{2mm}\par\medskip}

\newenvironment{proofof5}{\noindent {\bf
Proof of  Theorem 1.3.}}{\rule{2mm}{2mm}\par\medskip}

\newcommand{\remark}{\medskip\par\noindent {\bf Remark.~~}}
\newcommand{\pp}{{\it p.}}
\newcommand{\de}{\em}

\title{  {Extensions of Brunn-Minkowski's inequality to multiple matrices}
\thanks{This paper was firstly announced on Feb, 2020, and was later published on 
Linear Algebra and its Applications 603 (2020) 91--100.  
See \url{https://doi.org/10.1016/j.laa.2020.05.037}. 
This work was supported by  NSFC (Grant Nos. 11671402, 11871479). 
 E-mail addresses: ytli0921@hnu.edu.cn (Y. Li), 
fenglh@163.com (L. Feng, corresponding author).} }

 \date{May 13, 2020}

\author{Yongtao Li$^{a}$,  Lihua Feng$^{\dag,b}$\\
{\small ${}^a$School of Mathematics, Hunan University} \\
{\small Changsha, Hunan, 410082, P.R. China } \\
{\small $^b$School of Mathematics and Statistics, Central South University} \\
{\small New Campus, Changsha, Hunan, 410083, P.R. China. } }

\maketitle

\vspace{-0.5cm}

\begin{abstract}
Yuan and Leng (2007) gave a generalization of  Ky Fan's determinantal inequality, 
which is a celebrated refinement of the fundamental Brunn-Minkowski inequality 
$(\det (A+B))^{1/n} \ge (\det A)^{1/n} +(\det B)^{1/n}$, where $A$ and $B$ are positive semidefinite 
matrices. In this note, we first give an extension of Yuan-Leng's result 
to multiple positive definite matrices, 
and then we further extend the result to a larger class of matrices 
whose numerical ranges are contained in a sector. 
Our result improves a recent result of Liu [Linear Algebra Appl. 508 (2016) 206--213].
 \end{abstract}


\section{Introduction}

\label{sec1} 

We use the following standard notation. 
The set of $n\times n$ complex matrices is denoted by $\mathbb{M}_n(\mathbb{C})$, 
or simply by $\mathbb{M}_n$, 
and the identity matrix of order $n$ by  $I_n$, or $I$ for short. 
By convention, if $X\in \mathbb{M}_n$ is positive semidefinite, we write $X\ge 0$. 
For two Hermitian matrices $A$ and $B$ of the same size, $A\ge B$ means $A-B\ge 0$.

If $A$ and $B$ are positive semidefinite matrices of the same size, it is well-known that 
\begin{equation} \label{eqq1}
\det (A+B) \ge \det A +\det B.
\end{equation}

Over the years, 
various extensions and generalizations of (\ref{eqq1}) have been obtained in the literature;  
see, e.g., \cite{Hay70,Har73,Lin14} and updated results \cite{DW20,Mao20}. 
A fundamental generalization of (\ref{eqq1}) is the renowned Brunn-Minkowski determinantal inequality 
\cite[p. 510]{HJ13}, 
precisely, if $A$ and $B$ are $n$-square positive semidefinite, then 

\begin{equation} \label{eqbm}
(\det (A+B))^{1/n} \ge (\det A)^{1/n} +(\det B)^{1/n}. 
\end{equation}

Moreover, Ky Fan  proved the following extension of inequality (\ref{eqbm}) 
with respect to leading principal submatrix; 
see \cite{Fan55} or \cite[p. 687]{MOA11} for more details. 

\begin{theorem} (see \cite{Fan55}) \label{thmfan}
Let  $A$ and $B$ be positive definite matrices of size $n$. Then 
\begin{equation} \label{eqfan}
\left( \frac{\det (A+B)}{\det (A_k+B_k)} \right)^{\frac{1}{n-k}} \ge 
\left( \frac{\det A}{\det A_k} \right)^{\frac{1}{n-k}} +
\left( \frac{\det B}{\det B_k} \right)^{\frac{1}{n-k}},  
\end{equation}
where  $A_k$  denotes the $k$th leading principal submatrix of $A$. 
\end{theorem}

\noindent
{\bf Remark.} We remark that inequality (\ref{eqfan}) implies the following inequality,  
\begin{equation} \label{eq3}
\frac{\det (A+B)}{\det (A_k+B_k)} \ge \frac{\det A}{\det A_k} +
\frac{\det B}{\det B_k}, 
\end{equation}
which usually contributes to Bergstr\"{o}m \cite{Ber52}; 
also see, e.g., \cite[Theorem 3.1]{LM00}. 

Furthermore, Yuan and Leng \cite{YL07} gave the following extension (\ref{eqq5}) of 
Ky Fan's inequality (\ref{eqfan}). 
Both (\ref{eqfan}) and (\ref{eqq5}) can be viewed as Brunn-Minkowski type inequalities.

\begin{theorem} (see \cite{YL07}) \label{thm12}
Let $A$ and $B$ be positive definite matrices of size $n$.  
If $a$ and $b$ are two nonnegative real numbers such that $A\ge aI_n$ and $B\ge bI_n$, 
then 
\begin{equation} \label{eqq5}
\begin{aligned}
&\left( \frac{\det (A+B)}{\det (A_k+B_k)} -\det \bigl( (a+b)I_{n-k}\bigr)\right)^{\frac{1}{n-k}} \\
&\quad \ge \left( \frac{\det A}{\det A_k}-\det (aI_{n-k}) \right)^{\frac{1}{n-k}}  
 + \left( \frac{\det B}{\det B_k}-\det (bI_{n-k}) \right)^{\frac{1}{n-k}}. 
\end{aligned}
\end{equation}
\end{theorem}

As Liu \cite{Liu16} recently pointed out, 
a careful observation of Ky Fan's result (\ref{eqfan}) and using Bellman's inequality 
could actually give a simple proof of (\ref{eqq5}); 
see \cite{Liu16} for more details. 
On the other hand, Liu \cite{Liu16} also gave an extension of Ky Fan's inequality 
(\ref{eqfan}) and Brunn-Minkowski type inequality (\ref{eqq5}) to the case of sector matrices. 
Before stating Liu's result, we  need to introduce some required definitions and notations.

For $A\in \mathbb{M}_n$, the Cartesian (Toeptliz) decomposition 
$A=\Re A+i\Im A$, where $\Re A=\frac{1}{2}(A+A^*)$ and $\Im A=\frac{1}{2i}(A-A^*)$. 
Let $|A|$ denote the positive square root of $A^*A$, 
i.e., $|A|=(A^*A)^{1/2}$.  We denote the $i$-th largest singular value of $A$ by 
$s_i(A)$, then $s_i(A)=\lambda_i(|A|)$, 
the $i$-th largest eigenvalue of $|A|$. 
If $A=\begin{bmatrix}A_{11} & A_{12} \\ A_{21} & A_{22}  \end{bmatrix}$ 
is a square matrix with $A_{11}$  nonsingular, then the {\it Schur complement} of $A_{11}$ in $A$ 
is defined as $A/A_{11}:=A_{22}-A_{21}A_{11}^{-1}A_{12}$. 
It is obvious that $\det A=(\det A_{11})\det (A/A_{11})$. 
We refer to the integrated survey \cite{Zhang05} for more applications of Schur complement. 

Recall that the numerical range of $A\in \mathbb{M}_n$ is defined as 
\[ W(A)=\{x^*Ax : x\in \mathbb{C}^*,x^*x=1\}. \]
For $\alpha \in [0,\frac{\pi}{2})$, let $S_{\alpha}$ be the sector on the complex plane given by 
\[ S_{\alpha}=\{z\in \mathbb{C}: \Re z>0,|\Im z|\le (\Re z)\tan \alpha \} 
=\{re^{i\theta } : r>0,|\theta |\le \alpha \}. \]
Obviously, if $W(A)\subseteq S_{\alpha}$ for $\alpha \in [0,\frac{\pi}{2})$, 
then $\Re (A)$ is positive definite and if $W(A)\subseteq S_0$, 
then $A$ is positive definite. 
Such class of matrices whose numerical ranges are contained in a sector 
is called the {\it sector matrices class}. 
Clearly, the concept of sector matrices is an extension of that of positive definite matrices. 
Over the past years, 
various studies on sector matrices have been obtained in the literature; 
see, e.g., \cite{Choi19, Jiang19, Kua17, Lin15, YLC19, Zhang15} and reference therein. 

\vspace{0.3cm}

Liu's extension of Brunn-Minkowski type inequality (\ref{eqq5}) can be listed as below. 

\begin{theorem} (see \cite{Liu16}) \label{thm13}
Let $\alpha \in [0,\frac{\pi}{2})$ and $A, B$ be $n\times n$ matrices
 such that $W(A),W(B)\subseteq S_{\alpha}$. 
If $a$ and $b$ are two nonnegative numbers such that $\Re A \ge aI_n$ and $\Re B \ge bI_n$, then 
\begin{equation} \label{eqliu}
\begin{aligned}
& \left( \left|\frac{\det (A+B)}{\det (A_k+B_k)}\right| 
 - (\cos \alpha)^k\det \bigl( (a+b)I_{n-k}\bigr)\right)^{\frac{1}{n-k}} \\
&\quad \ge (\cos \alpha)^{\frac{n+k}{n-k}}
 \left( \left|\frac{\det A}{\det A_k} \right| 
- \frac{\det (aI_{n-k})}{(\cos \alpha)^n} \right)^{\frac{1}{n-k}}  \\
&\quad +(\cos \alpha)^{\frac{n+k}{n-k}} 
\left( \left|\frac{\det B}{\det B_k} \right| - \frac{\det (bI_{n-k})}{(\cos \alpha)^n}
 \right)^{\frac{1}{n-k}}. 
\end{aligned}
\end{equation}
\end{theorem}

In this paper, we will give the following improvement of  Theorem \ref{thm13}. 

\begin{theorem} \label{thm29}
Let $A_i\in \mathbb{M}_n$ with $W(A_i)\subseteq S_{\alpha}$ and  
$a_i$ be nonnegative real numbers such that  $\Re A_i \ge a_i I_n$ for every $i=1,2,\ldots ,m$, and let 
$A_{ik},k=1,2,\ldots ,n-1$ denote the $k$th leading principal submatrix of $A_i,i=1,2,\ldots ,m$. 
Then 
\begin{equation*}
\begin{aligned}
&\left( \left| \frac{\det (\sum_{i=1}^m A_i)}{\det (\sum_{i=1}^m A_{ik})} \right|- 
\det \biggl( \sum_{i=1}^m a_iI_{n-k} \biggr)\right)^{\frac{1}{n-k}}  \\ 
& \quad \ge 
 (\cos \alpha)^{\frac{n}{n-k}} \sum_{i=1}^m
\left( \left|\frac{\det A_i}{\det A_{ik}} \right| - 
\frac{\det (a_iI_{n-k})}{(\cos \alpha)^n} \right)^{\frac{1}{n-k}}. 
\end{aligned}
\end{equation*}
\end{theorem} 

Clearly, when $m=2$, Theorem \ref{thm29} leads to 
\begin{equation*} 
\begin{aligned}
& \left( \left|\frac{\det (A_1+A_2)}{\det (A_{1k}+A_{2k})}\right| 
 - \det \bigl( (a_1+a_2)I_{n-k}\bigr)\right)^{\frac{1}{n-k}} \\
&\quad \ge (\cos \alpha)^{\frac{n}{n-k}}
 \left( \left|\frac{\det A_1}{\det A_{1k}} \right| 
- \frac{\det (a_1I_{n-k})}{(\cos \alpha)^n} \right)^{\frac{1}{n-k}}  \\
&\quad +(\cos \alpha)^{\frac{n}{n-k}} 
\left( \left|\frac{\det A_2}{\det A_{2k}} \right| - \frac{\det (a_2I_{n-k})}{(\cos \alpha)^n}
 \right)^{\frac{1}{n-k}}, 
\end{aligned}
\end{equation*}
which is indeed an improvement of Liu's result (\ref{eqliu}) in Theorem \ref{thm13}.

The paper is organized as follows. 
We first extend Yuan-Leng's result (\ref{eqq5}) to multiple positive definite matrices 
(Theorem \ref{thm28}). 
Then we will show a proof of our main result (Theorem \ref{thm29}). 
Finally, we will provide some corollaries based on Theorem \ref{thm29}. 
Our results improve and generalize the above mentioned results (\ref{eqfan}), (\ref{eqq5}) 
and (\ref{eqliu}). 
To some extend, this paper 
could be regarded as a continuation and development of  \cite{Mao20}.

\section{Auxiliary results and proofs}
\label{sec2}

First, we list four lemmas which are useful to establish our extension (Theorem \ref{thm29}). 
 The first lemma is known as the Ostrowski-Taussky inequality (see \cite[p. 510]{HJ13}).

\begin{lemma}  \label{lem22}
Let $A$ be an $n$-square complex matrix. Then 
\[ \lambda_i(\Re A) \le s_i(A),\quad i=1,2,\ldots ,n. \]
Moreover, if $\Re A$ is positive definite, then 
\[ \det \Re A + |\det \Im A| \le |\det A|. \]
\end{lemma}

The second lemma \cite{Lin15} gives a reverse of the Ostrowski-Taussky inequality. 

\begin{lemma}  \label{lem21}
Let $0\le \alpha <\frac{\pi}{2}$ and 
$A\in \mathbb{M}_n$ with $W(A)\subseteq S_{\alpha}$. Then 
\[ |\det A| \le (\sec \alpha)^n \det (\Re A). \]
\end{lemma}

The third and fourth lemma reveal more informations for leading principal submatrix and its 
Schur complement; see \cite[Proposition 2.1 and Lemma 2.5]{Lin15} for more details. 

\begin{lemma}  \label{lem23}
Let $0\le \alpha <\frac{\pi}{2}$ and 
$A=\begin{bmatrix}A_{11} & A_{12} \\ A_{21} & A_{22} \end{bmatrix}\in \mathbb{M}_n$  
with diagonal blocks $A_{11},A_{22}$ being square.  
If $W(A)\subseteq S_{\alpha}$, then $W(A_{11})\subseteq S_{\alpha}$ 
and $W(A/A_{11})\subseteq S_{\alpha}$. 
\end{lemma}

\begin{lemma} \label{lem24}
Let $A=\begin{bmatrix}A_{11} & A_{12} \\ A_{21} & A_{22} \end{bmatrix}\in \mathbb{M}_n$  
with diagonal blocks $A_{11},A_{22}$ being square.
If $\Re A$ is positive definite, then 
\[ \Re (A/A_{11}) \ge (\Re A)/(\Re A_{11}). \]
\end{lemma}

\begin{corollary} \label{coro25}
Let A be a square matrix with $\Re A$ being positive definite. Then  
\[ \left| \frac{\det A}{\det A_k} \right| \ge \frac{\det (\Re A)}{\det (\Re A_k)}, \]
where $A_k$ stands for the $k$th leading principal submatrix of $A$.
\end{corollary}

\begin{proof}
By Lemma \ref{lem22}, Lemma \ref{lem23} and Lemma \ref{lem24}, one can get 
\[ \left| \frac{\det A}{\det A_k} \right| =|\det (A/A_{k})| \ge 
\det (\Re (A/A_k)) \ge \det ((\Re A)/(\Re A_{11})) =\frac{\det (\Re A)}{\det (\Re A_k)}.  \] 
This completes the proof. 
\end{proof}

Next, we will present two propositions to facilitate the proofs of our main results.

\begin{proposition} \label{prop26}
Let $x_{ij},i=1,2,\ldots ,m,j=1,2,\ldots ,n$ be nonnegative real numbers. 
If $p\ge 1$ and $x_{i1}^p \ge \sum_{j=2}^n x_{ij}^p$ for each $i=1,2,\ldots ,m$, then
\[ \left( \left( \sum_{i=1}^m x_{i1}\right)^p - 
\sum_{j=2}^n \left( \sum_{i=1}^m x_{ij}\right)^p \right)^{1/p} \ge 
\sum_{i=1}^m \left( x_{i1}^p - \sum_{j=2}^n x_{ij}^p\right)^{1/p}. \]
\end{proposition}

\begin{proof}
The required inequality is equivalent to 
\[ \left(\sum_{i=1}^m \biggl( x_{i1}^p - \sum_{j=2}^n x_{ij}^p\biggr)^{1/p}\right)^p + 
\sum_{j=2}^n \left( \sum_{i=1}^m x_{ij}\right)^p \le \left( \sum_{i=1}^m x_{i1}\right)^p. \]
The remarkable Minkowski inequality states that if $f_1,f_2,\ldots ,f_m \in \mathbb{R}^n$, then 
\begin{equation} \label{eqmin}
 \lVert f_1+f_2+\cdots +f_m\rVert \le \lVert f_1 \rVert + 
\lVert f_2 \rVert +\cdots +\lVert f_m \rVert ,
\end{equation}
where $\lVert \cdot \rVert$ stands for the $p$-norm on $\mathbb{R}^n$. 
By setting 
\[ f_i:=\Bigl( \bigl( x_{i1}^p-\sum_{j=2}^nx_{ij}^p\bigr)^{1/p}, 
x_{i2},\ldots ,x_{in}\Bigr)\in \mathbb{R}^n \]
for each $i=1,2,\ldots ,m$ in (\ref{eqmin}), which leads to the desired inequality. 
\end{proof}

Clearly, when $m=2$, Proposition \ref{prop26} reduces to Bellman's inequality \cite[p. 38]{BB61}. 

The following Proposition \ref{prop27} is a direct extension of Ky Fan's inequality (\ref{eqfan}), 
which palys an essential role in our extension.

\begin{proposition} \label{prop27}
Let $A_j\in \mathbb{M}_n$ be positive definite and let 
$A_{jk},k=1,2,\ldots ,n-1$ denote the $k$th leading principal submatrix of $A_j,j=1,2,\ldots ,m$. Then 
\[   
\left( \frac{\det ( \sum_{j=1}^m A_j)}{\det ( \sum_{j=1}^m A_{jk})}\right)^{\frac{1}{n-k}} 
\ge \sum_{j=1}^m \left( \frac{\det A_j}{\det A_{jk}}\right)^{\frac{1}{n-k}}.
\]
\end{proposition}

\begin{proof}
The proof is by induction on $m$. 
When $m=1$, there is nothing to prove. 
We now prove the base case $m=2$. In this case, 
the desired inequality just is Ky Fan's ineqaulity (\ref{eqfan}). 
For convenience of readers, we here include a proof of (\ref{eqfan}). 
Without loss of generality, we may assume that 
\[ A=\begin{bmatrix}A_{11}  & A_{12} \\ A_{21} & A_{22}  \end{bmatrix} 
~~\text{and}~ B=\begin{bmatrix}B_{11}  & B_{12} \\ B_{21} & B_{22}  \end{bmatrix}, \]
where $A_{11}$ and $B_{11}$ are $k\times k$ submatrices. 
It is easy to see from  Schur complement  that 
\[ \begin{bmatrix}A_{11}  & A_{12} \\ A_{21} & A_{21}A_{11}^{-1}A_{12}  \end{bmatrix} \ge 0
~~\text{and}~ 
\begin{bmatrix}B_{11}  & B_{12} \\ B_{21} & B_{21}B_{11}^{-1}B_{12}  \end{bmatrix} \ge 0.\]
Therefore, we have 
\[ \begin{bmatrix}A_{11}+B_{11}  & A_{12} +B_{12} 
\\ A_{21} +B_{21}& A_{21}A_{11}^{-1}A_{12}+B_{21}B_{11}^{-1}B_{12}  \end{bmatrix} \ge 0. \]
By making use of Schur complement again, we get 
\[ A_{21}A_{11}^{-1}A_{12}+B_{21}B_{11}^{-1}B_{12} 
\ge (A_{21} +B_{21})(A_{11}+B_{11})^{-1}(A_{12} +B_{12}), \]
which is equivalent to 
\[ (A+B)/(A_{11}+B_{11})\ge A/ A_{11} +B/B_{11}, \]
which together with Brunn-Minkowski's inequality (\ref{eqbm}) leads to 
\[ (\det (A+B)/(A_{11}+B_{11}))^{1/(n-k)} \ge 
( \det A/ A_{11})^{1/(n-k)} +( \det B/ B_{11})^{1/(n-k)}. \]
Noting that $ \det (A/ A_{11}) =\frac{\det A}{\det A_{11}}$, this completes the proof of the case $m=2$. 

Suppose the proposition is true for $m$, and then we consider the case $m+1$, 
\begin{align*}
\left( \frac{\det ( \sum_{j=1}^{m+1} A_j)}{\det ( \sum_{j=1}^{m+1} A_{jk})}\right)^{\frac{1}{n-k}} 
&=\left( \frac{\det ( A_1+\sum_{j=2}^{m+1} A_j)}{\det ( 
A_{1k}+\sum_{j=2}^{m+1} A_{jk})}\right)^{\frac{1}{n-k}}  \\
&\ge \left( \frac{\det A_1}{\det A_{1k}}\right)^{\frac{1}{n-k}} + 
\left( \frac{\det ( \sum_{j=2}^{m+1} A_j)}{\det ( \sum_{j=2}^{m+1} A_{jk})}\right)^{\frac{1}{n-k}} \\
&\ge \sum_{j=1}^{m+1} \left( \frac{\det A_j}{\det A_{jk}}\right)^{\frac{1}{n-k}}, 
\end{align*}
where the first inequality follows from the base case, and the last one follows from 
induction hypothesis. 
Thus the proof of induction step is complete. 
\end{proof}

Now, we are ready to present our first main result. 
Clearly, when $m=2$, Theorem \ref{thm28} reduces to Yuan-Leng's result (Theorem \ref{thm12}); 
When $m=2$ and $a_i=0$, Theorem \ref{thm28} becomes Ky Fan's inequality (Theorem \ref{thmfan}).

\begin{theorem} \label{thm28}
Let $A_i\in \mathbb{M}_n$ be positive definite and $a_i$ be nonnegative real numbers such that 
$A_i\ge a_iI_n$ for every $i=1,2,\ldots ,m$, and let 
$A_{ik},k=1,2,\ldots ,n-1$ denote the $k$th leading principal submatrix of $A_i,i=1,2,\ldots ,m$. Then 
\begin{equation*}
\left( \frac{\det (\sum_{i=1}^m A_i)}{\det (\sum_{i=1}^m A_{ik})} - 
\det \biggl( \sum_{i=1}^m a_iI_{n-k}\biggr)\right)^{\frac{1}{n-k}} \ge 
\sum_{i=1}^m \left( \frac{\det A_i}{\det A_{ik}}-\det (a_iI_{n-k}) \right)^{\frac{1}{n-k}}. 
\end{equation*}
\end{theorem}

\begin{proof}
Setting $p=n-k,n=2$ and $x_{i1}=\bigl(\frac{\det A_i}{\det A_{ik}}\bigr)^{1/(n-k)}, x_{i2}=a_i$ in 
Proposition \ref{prop26}. 
Since $A_i\ge a_iI_n$, it is easy to see from (\ref{eq3}) that  for each $i$, 
\[ \frac{\det A_i}{\det A_{ik}} \ge  
\frac{\det (a_iI_n)}{\det (a_iI_k)} +\frac{\det (A_i-a_iI_n)}{\det (A_{ik}-a_iI_k)} \ge a_i^{n-k}.\]
 That is, all conditions in Proposition \ref{prop26} are satisfied. 
Therefore, we get 
\begin{align*}
&\sum_{i=1}^m \left( \frac{\det A_i}{\det A_{ik}}-\det (a_iI_{n-k}) \right)^{\frac{1}{n-k}} \\
&\quad \le \left( \biggl( \sum_{i=1}^m \Bigl( \frac{\det A_i}{\det A_{ik}}\Bigr)^{
\frac{1}{n-k}} \biggr)^{n-k} - \biggl( \sum_{i=1}^m a_i\biggr)^{n-k} \right)^{\frac{1}{n-k}} \\
&\quad \le \left( \frac{\det (\sum_{i=1}^m A_i)}{\det (\sum_{i=1}^m A_{ik})} - 
\det \biggl( \sum_{i=1}^m a_iI_{n-k}\biggr)\right)^{\frac{1}{n-k}}, 
\end{align*}
where the last inequality follows from Proposition \ref{prop27}. 
\end{proof}

Now, we are ready to present a proof of Theorem \ref{thm29}. 

\vspace{0.3cm}

\noindent
{\bf Proof of Theorem \ref{thm29}.}~
Since $\Re (\sum_{i=1}^m A_i)$ is positive definite, by Corollary \ref{coro25}, we obtain 
\begin{align*}
& \left( \left| \frac{\det (\sum_{i=1}^m A_i)}{\det (\sum_{i=1}^m A_{ik})} \right|- 
\det \biggl( \sum_{i=1}^m a_iI_{n-k} \biggr)\right)^{\frac{1}{n-k}} \\
&\quad \ge 
\left(  \frac{\det (\Re \sum_{i=1}^m A_i)}{\det (\Re \sum_{i=1}^m A_{ik})} - 
\det \biggl( \sum_{i=1}^m a_iI_{n-k} \biggr)\right)^{\frac{1}{n-k}} \\
&\quad \ge 
\sum_{i=1}^m \left( \frac{\det \Re A_i}{\det \Re A_{ik}}-\det (a_iI_{n-k}) \right)^{\frac{1}{n-k}} \\
& \quad \ge 
\sum_{i=1}^m \left( (\cos \alpha)^n\left|\frac{\det A_i}{\det A_{ik}}\right| 
-\det (a_iI_{n-k}) \right)^{\frac{1}{n-k}} \\
&\quad = (\cos \alpha)^{\frac{n}{n-k}} \sum_{i=1}^m 
\left( \left|\frac{\det A_i}{\det A_{ik}} \right| - 
\frac{\det (a_iI_{n-k})}{(\cos \alpha)^n} \right)^{\frac{1}{n-k}}, 
\end{align*}
where the second inequality follows from Theorem \ref{thm28}, and 
the last one follows from Lemma \ref{lem21} and Lemma \ref{lem22}.

\vspace{0.3cm}

At the end of the paper, 
we will provide some corollaries of Theorem \ref{thm29}. 
Although the following corollary has its root in \cite{Liu16}, 
 our result improves the inequality stated in \cite{Liu16}. 
For $A\in \mathbb{M}_n$, $A$ is called {\it accretive-dissipative} if both $\Re A$ and 
$\Im A$ are positive definite (see \cite{GI05}). 
This class of matrices recently has attracted  great attentions; 
see, e.g., \cite{Lin13,GHK17,Kit18}. 
Observe that if $A$ is accretive-dissipative, 
then $W(e^{-i\pi/4}A)\subseteq S_{\pi /4}$, 
which leads to the following Corollary \ref{coro}.

\begin{corollary} \label{coro}
Let $A_i\in \mathbb{M}_n$ be accretive-dissipative matrices and  
$a_i$ be nonnegative real numbers such that  $\Re A_i \ge a_i I_n$ for every $i=1,2,\ldots ,m$, and let 
$A_{ik},k=1,2,\ldots ,n-1$ denote the $k$th leading principal submatrix of $A_i,i=1,2,\ldots ,m$. 
Then 
\begin{equation*}
\begin{aligned}
&\left( \left| \frac{\det (\sum_{i=1}^m A_i)}{\det (\sum_{i=1}^m A_{ik})} \right|- 
\det \biggl( \sum_{i=1}^m a_iI_{n-k} \biggr)\right)^{\frac{1}{n-k}}  \\ 
& \quad \ge 
 \frac{1}{2^{{n}/{2(n-k)}}} \sum_{i=1}^m
\left( \left|\frac{\det A_i}{\det A_{ik}} \right| - 2^{n/2}\det (a_iI_{n-k})\right)^{\frac{1}{n-k}}. 
\end{aligned}
\end{equation*}
\end{corollary}

By replacing $A_i$ with $\lambda_iA_i$  and 
meanwhile replacing $a_i$ with $\lambda_ia_i$ for each $i=1,2,\ldots,m$, respectively. 
Theorem \ref{thm29} immediately yields the following Corollary \ref{coro210}. 

\begin{corollary} \label{coro210}
Let $A_i\in \mathbb{M}_n$ with $W(A_i)\subseteq S_{\alpha}$ and  
$a_i$ be nonnegative real numbers such that  $\Re A_i \ge a_i I_n$ for every $i=1,2,\ldots ,m$, and let 
$A_{ik},k=1,2,\ldots ,n-1$ denote the $k$th leading principal submatrix of $A_i,i=1,2,\ldots ,m$. 
Then 
\begin{equation*}
\begin{aligned}
&\left( \left| \frac{\det (\sum_{i=1}^m \lambda_iA_i)}{\det (\sum_{i=1}^m \lambda_iA_{ik})} \right|- 
\det \biggl( \sum_{i=1}^m \lambda_ia_iI_{n-k} \biggr)\right)^{\frac{1}{n-k}}  \\ 
& \quad \ge 
 (\cos \alpha)^{\frac{n}{n-k}}  \sum_{i=1}^m \lambda_i
\left( \left|\frac{\det A_i}{\det A_{ik}} \right| - 
\frac{\det (a_iI_{n-k})}{(\cos \alpha)^n} \right)^{\frac{1}{n-k}} 
\end{aligned}
\end{equation*}
holds for all nonnegative mumbers $\lambda_1,\lambda_2,\ldots ,\lambda_m$.
\end{corollary}

By applying the weighted arithmetic-geometric mean inequality to the 
right hand side of Corollary \ref{coro210}, one could get  

\begin{corollary} \label{coro211}
Let $A_i\in \mathbb{M}_n$ with $W(A_i)\subseteq S_{\alpha}$ and  
$a_i$ be nonnegative real numbers such that  $\Re A_i \ge a_i I_n$ for every $i=1,2,\ldots ,m$, and let 
$A_{ik},k=1,2,\ldots ,n-1$ denote the $k$th leading principal submatrix of $A_i,i=1,2,\ldots ,m$. 
Then 
\begin{equation*}
\begin{aligned}
& \left| \frac{\det (\sum_{i=1}^m \lambda_iA_i)}{\det (\sum_{i=1}^m \lambda_iA_{ik})} \right|- 
\det \biggl( \sum_{i=1}^m \lambda_ia_iI_{n-k} \biggr)  \\ 
& \quad \ge 
 (\cos \alpha)^{n}  \prod_{i=1}^m 
\left( \left|\frac{\det A_i}{\det A_{ik}} \right| - 
\frac{\det (a_iI_{n-k})}{(\cos \alpha)^n} \right)^{\lambda_i} 
\end{aligned}
\end{equation*}
holds for all nonnegative mumbers $\lambda_1,\lambda_2,\ldots ,\lambda_m$ with 
$\sum_{i=1}^m \lambda_i=1$.
\end{corollary}

Setting $k=0$ in Corollary \ref{coro211}, we get the following generalization of Ky Fan's 
determinantal inequality (see  \cite[p. 488]{HJ13}), 
i.e., the log-concavity of the determinant functional over the cone of positive semidefinite matrices.

\begin{corollary}
Let $A_i\in \mathbb{M}_n$ with $W(A_i)\subseteq S_{\alpha}$ and  
$a_i$ be nonnegative real numbers such that  $\Re A_i \ge a_i I_n$ for every $i=1,2,\ldots ,m$. 
If $\lambda_1,\lambda_2,\ldots ,\lambda_m$ are nonnegative numbers with $\sum_{i=1}^m \lambda_i=1$,  
then 
\begin{align*}& \left|\det \biggl( \sum_{i=1}^m \lambda_i A_i\biggr)\right| - 
\biggl( \sum_{i=1}^m \lambda_i a_i\biggr)^n 
\ge (\cos \alpha)^n \prod_{i=1}^m \biggl( |\det A_i| - \frac{a_i^n}{(\cos \alpha)^n}\biggr)^{
\lambda_i}. 
\end{align*}
\end{corollary}

\section*{Acknowledgments}
This work was supported by  NSFC (Grant Nos. 11671402, 11871479).

\end{document}